\documentclass[12pt,reqno]{amsart}
\usepackage{amsmath}
\usepackage{eucal}
\usepackage{graphicx}
\newtheorem{defi}{Definition}
\newcommand{\brdef}{\begin{defi}}
\newcommand{\erdef}{\end{defi}}

\newtheorem{cor}{Corollary}
\newcommand{\bcor}{\begin{cor}}
\newcommand{\ecor}{\end{cor}}

\newtheorem{thm}{Theorem}
\newcommand{\bth}{\begin{thm}}
\newcommand{\eth}{\end{thm}}
\newtheorem{lem}{Lemma}
\newcommand{\ble}{\begin{lem}}
\newcommand{\ele}{\end{lem}}

\def\pn{\par\noindent}

 \huge

\usepackage{hyperref}
\usepackage{mathrsfs}

\numberwithin{equation}{section}

\begin{document}
\begin{center}
{\large \bf Certain Curvature Conditions on \\
$N(k)$-Paracontact Metric Manifolds} \\
\
\
\\
{Vishnuvardhana. S.V.\footnote{ corresponding author.}$^{a}$, Venkatesha$^{b}$, B. Phalaksha Murthy$^{b}$ and B. Shanmukha$^{b}$}\\
$^{a}$ Department of Mathematics, GITAM School of Technology,\\
GITAM(Deemed to be university), Bangalore, Karnataka, INDIA.\\
$^{b}$ Department of Mathematics, Kuvempu University,\\
Shankaraghatta - 577 451, Shimoga, Karnataka, INDIA.\\
\pn{\tt e-mail:{\verb+svvishnuvardhana@gmail.com+, \verb+vensmath@gmail.com+, \verb+pmurthymath@gmail.com+, \verb+meshanmukha@gmail.com+}}

\end{center}
\begin{quotation}
{{\bf Abstract}: The aim of the present paper is to study pseudo-symmetric, Ricci generalized pseudo-symmetric and generalized Ricci recurrent $N(k)$-Paracontact Metric Manifolds.}
\end{quotation}
\section{{\bf Introduction \label{sec1}}}

In modern geometry one of the most interesting research topics are on contact and paracontact geometry. If we look at the recent developments in these topics, there is an impression that geometers are more concentrated in the study of nullity distribution on contact and paracontact manifolds by emphasizing similarities and differences between them. The study of paracontact geometry was triggered by Kaneyuki and Kozai \cite{SKMK}. An efficient contribution to this geometry was given by Zamkovoy \cite{SZ}, Kaneyuki \cite{SKFLW}, Alekseevsky et al. \cite{DVACM}, etc. A significant subclass of paracontact metric manifold like para-Sasakian manifold was introduced by Zamkovoy \cite{SZ}. A normal paracontact metric manifold is called a para-Sasakian manifold and which implies a K-paracontact condition and the converse holds only in dimension 3. In any para-Sasakian manifold
\begin{equation}
\label{2.3a} \mathcal{R}(X, Y)\xi =-\{\eta(Y)X - \eta(X)Y\},
\end{equation}
holds. Some of differences between contact and paracontact cases are: Unlike in contact metric geometry the condition (\ref{2.3a}) does not imply that the paracontact manifold is paraSasakian.  Another important difference between them is due to the non-positive definiteness of the metric.

\par The nullity distribution on paracontact manifolds was introduced by Montano et al.\cite{BCMIKECM}. Molina and his co-author [\cite{VMM}, \cite{GCVMM}] obtained some examples and classification theorems on paracontact metric $(k,\mu)$-spaces. The main difference between contact metric $(\tilde{k},\tilde{\mu})$-spaces and paracontact metric $(k, \mu)$-spaces is that, the constant $\tilde{k}$ cannot be greater than 1 incase of contact metric $(\tilde{k},\tilde{\mu})$-spaces but, no restrictions for constants $k$ and $\mu$ incase of paracontact metric $(k,\mu)$-spaces.

After introducing torse forming vector fields by Yano \cite{KYano}, most of the geometers studied these on different manifolds with different curvature restrictions because of their applications in many branches of physics.

\par With this background, in this article we study some curvature properties of $N(k)$-paracontact metric manifolds. The paper is organized as follows: After preliminaries in Section \ref{sec2}, we proved that a torse forming vector field in a 3-dimensional $N(k)$-paracontact metric manifold $\mathcal{M}^{3}$ is a concircular vector field. In the next section we have shown that a non-flat $N(k)$-paracontact metric manifold is a proper pseudo-symmetric manifold then the manifold is a pseudo-symmetric manifold of constant type. Section \ref{sec5} deals with the study of Ricci generalized pseudo-symmetric $N(k)$-paracontact metric manifolds. We prove that in a generalized Ricci recurrent $N(k)$-paracontact metric manifold $\mathcal{M}$, the associated 1-forms are linearly dependent and the vector fields of the associated 1-forms are of opposite direction in section \ref{sec6}. In section \ref{sec7}, we constructed an example to verify our some of the results.

%
%

\section{{\bf Preliminaries \label{sec2}}}
A smooth manifold $\mathcal{M}^{2n+1}$ is said to have an almost paracontact structure if it admits a $(1, 1)$-tensor field $\phi$, a vector field $\xi$ and a 1-form $\eta$ satisfying following conditions \cite{SKFLW}:
\begin{eqnarray}
\label{2.1} \phi^{2}=I-\eta\otimes\xi,\,\,\,\,\,\,\,\,\,\,\eta(\xi)=1,\,\,\,\,\,\,\,\,\,\,\,\phi\xi=0,\,\,\,\,\,\,\,\,\,\,\,\eta\circ\phi=0.
\end{eqnarray}

\par A pseudo-Riemannian metric $g$ with almost paracontact structure $(\phi,\xi,\eta)$ such that,
\begin{eqnarray}
\label{2.2} g(\phi X,\phi Y)= -g(X,Y) + \eta(X)\eta(Y),\,\,\,\,\,\,\,\,\,g(X,\xi)=\eta(X), \,\,\,\,\,\,\,\,\, g(\phi X, Y)=-g(X, \phi Y),
\end{eqnarray}
then the structure $(\phi,\xi,\eta, g)$ on $\mathcal{M}^{2n+1}$ is said be almost paracontact metric structure. A manifold $\mathcal{M}^{2n+1}$ together with this almost paracontact metric structure is called an almost paracontact metric manifold and it is denoted by $\mathcal{M}^{2n+1}(\phi,\xi,\eta,g)$. For an almost paracontact metric manifold, there always exists a $\phi$-basis.

A paracontact metric $(k, \mu)$-manifold \cite{BCMLDT} is a paracontact metric manifold for which the curvature tensor field satisfies
\begin{eqnarray}
\nonumber \mathcal{R}(X,Y)\xi=k\{\eta(Y)X -\eta(X)Y\} + \mu\{\eta(Y)hX - \eta(X)hY\},
\end{eqnarray}
for all $X,Y\in \mathcal{TM}$, where $k$, $\mu\in R$. Here $2h$ is the Lie derivative of $\phi$ in the direction of $\xi$. Moreover, a symmetric, trace-free $(1, 1)$-tensor field $h$ satisfies.
\begin{eqnarray}
\label{2.3} h\xi=0,\,\,\,\,\,\,\,\,h\phi+\phi h=0,\,\,\,\,\,\,\,\,\,\nabla_{X}\xi=-\phi X + \phi hX.
\end{eqnarray}
Furthermore, $h=0$ is proportionate to, $\xi$ being killing and in such cases we call $\mathcal{M}^{2n+1}(\phi,\xi,\eta,g)$ as a K-paracontact manifold.

If $\mu= 0$, the $(k,\mu)$-paracontact metric manifold reduces to $N(k)$-paracontact metric manifold. Thus, for an $N(k)$-paracontact metric manifold we have
\begin{equation}
\label{2.3b} R(X,Y)\xi=k\{\eta(Y)X-\eta(X)Y\},
\end{equation}
$k$ being a constant.

In a $N(k)$-paracontact metric manifold, following relations hold \cite{DEBJSKMMT}:
\begin{eqnarray}
\label{2.4} & & h^{2}=(1+k)\phi^{2},\\
\label{2.5} & & (\nabla_{X}\eta)(Y) = g(X, \phi Y ) - g(hX, \phi Y),\\
\label{2.6} & & R(\xi, X)Y=k\{g(X, Y)\xi - \eta(Y)X\},\\
\label{2.10} & & S(X,\xi)=2nk\eta(X).
\end{eqnarray}

\section{\bf Torse forming vector field on 3-dimensional $N(k)$-paracontact metric manifolds \label{sec3}}
\begin{defi}
On a pseudo-Riemannian manifold, if covariant derivative of a vector field $\upsilon$ satisfies
\begin{equation}
(\nabla_{X}\omega)(Y)= \rho g(X, Y) + \beta(X)\omega(Y), \label{3.1}
\end{equation}
then $\upsilon$ is called as torse forming vector field. Here, $\upsilon$ is defined as $g(X, \upsilon) = \omega(X)$ for any vector field $X$, $\rho$ is a non-zero scalar and $\beta$ is a non-zero 1-form.
\end{defi}

Let us consider a $N(k)$-paracontact metric manifold $\mathcal{M}$ admitting a unit torse forming vector field $\widetilde{\upsilon}$ corresponding to the non-null torse forming vector field $\upsilon$. Hence if $T(X) = g(X, \widetilde{\upsilon})$, then we have
\begin{equation}
T(X)=\frac{\omega(X)}{\sqrt{\omega(\upsilon)}}. \label{3.2}
\end{equation}

From (\ref{3.1}) and (\ref{3.2}), we get
\begin{equation}
(\nabla_{X}T)(Y)=\lambda g(X, Y) + \beta(X)T(Y), \label{3.3}
\end{equation}
where $\lambda=\frac{\rho}{\sqrt{\omega(\upsilon)}}$

Since $\widetilde{\upsilon}$ is a unit vector field,
\begin{equation}
a)\,\,\,\beta(X)=-\lambda T(X)\,\,\,\,\,\,\,\,\,\,\,\,\,\,\,\,\,\,\,\,\,\,b)\,\,\,(\nabla_{X}T)(Y)=\lambda [g(X, Y) - T(X)T(Y)], \label{3.4}
\end{equation}
which implies that the 1-form $T$ is closed. Now differentiating (\ref{3.4})(b) covariantly and using
the Ricci identity, we obtain
\begin{eqnarray}
\nonumber T(\mathcal{R}(X, Y)Z) = (Y\lambda)[g(X, Z) - T(X)T(Z)] - (X\lambda)[g(Y, Z) - T(Y)T(Z)] \\
+ \lambda^{2}[g(X, Z)T(Y) - g(Y, Z)T(X) ]. \label{3.5}
\end{eqnarray}

Replace $Z$ by $\xi$ and then using (\ref{2.3b}) and $T(\xi)=\eta(\widetilde{\upsilon})$
\begin{eqnarray}
\nonumber (k+\lambda^{2})\{\eta(Y)T(X)-\eta(X)T(Y)\} + [X(\lambda)\eta(Y)-Y(\lambda)\eta(X)] \\
+ \eta(\widetilde{\upsilon})\{Y(\lambda)T(X)-X(\lambda)T(Y)\}=0. \label{3.6}
\end{eqnarray}

Put $X$ by $\widetilde{\upsilon}$ in (\ref{3.6}) and using $T(\widetilde{\upsilon})=1$ we get
\begin{eqnarray}
(k+\lambda^{2}+\widetilde{\upsilon}(\lambda))\{\eta(Y)-\eta(\widetilde{\upsilon})T(Y)\} =0, \label{3.7}
\end{eqnarray}
which gives either
\begin{eqnarray}
k+\lambda^{2}+\widetilde{\upsilon}(\lambda)=0 \label{3.8} \\
or\,\,\,\,\, \eta(Y)-\eta(\widetilde{\upsilon})T(Y)=0. \label{3.9}
\end{eqnarray}
Suppose equation (\ref{3.9}) (i.e., equation (\ref{3.8}) not holds) holds. Putting $X = \xi$ in (\ref{3.9}), we have $\eta(\widetilde{\upsilon}) = \pm 1$. This implies that
\begin{equation}
\eta(X) = \pm T(X). \label{3.10}
\end{equation}
From (\ref{2.3}), (\ref{3.4}) and (\ref{3.10}), we get $\lambda=\pm p(\text{a constant})$. Hence the vector field $\widetilde{\upsilon}$ is concircular.

Now, suppose equation (\ref{3.8}) (i.e., equation (\ref{3.9}) not holds) holds. Putting $X=\xi$ and then contraction of (\ref{3.5}) gives
\begin{equation}
\xi(\lambda)=\widetilde{\upsilon}(\lambda)\eta(\widetilde{\upsilon}). \label{3.11}
\end{equation}

Putting $Y=\xi$ in (\ref{3.6}), we get
\begin{equation}
X(\lambda)=-(k+\lambda^{2})T(X).   \label{3.12}
\end{equation}

Using (\ref{3.2})(a) and the above equation, one can get
\begin{eqnarray}
Y(\beta(X))=(k+\lambda^{2})T(X)T(Y)-\lambda Y(T(X)), \label{3.13}\\
X(\beta(Y))=(k+\lambda^{2})T(X)T(Y)-\lambda X(T(Y)).    \label{3.14}
\end{eqnarray}
From (\ref{3.13}) and (\ref{3.14}), we have
\begin{equation}
\beta[X, Y]=-\lambda T[X, Y], \,\,\,\,\,\,\,\,\,\,\,\, d\beta[X, Y]=-\lambda [(dT)(X, Y)]. \label{3.15}
\end{equation}
Since T is closed, $\beta$ is also closed which implies the vector field $\widetilde{\upsilon}$ is concircular. Thus we have
\begin{thm}\label{T1}
 A torse forming vector field in a 3-dimensional $N(k)$-paracontact metric manifold $\mathcal{M}^{3}$ is a concircular vector field.
\end{thm}

\section{\bf Pseudo-symmertic $N(k)$-paracontact metric manifolds \label{sec4}}
The curvature tensor $\mathcal{R}$ satisfies the condition
\begin{equation}
\mathcal{R}(X, Y)\cdot \CMcal{I}=L_{\CMcal{I}}[(X\wedge_{g} Y)\cdot \CMcal{I}], \label{I1}
\end{equation}
at every point of the Riemannian manifold then the manifold $M$ is called pseudo-symmetric (resp., Ricci-pseudo-symmetric) manifold when $\CMcal{I}=\mathcal{R} (resp., \mathcal{S})$. Here $(X\wedge_{g} Y)$ is an endomorphism and is defined by
\begin{equation}
\label{I2}   (X\wedge_{g} Y)Z=g(Y, Z)X-g(X, Z)Y,
\end{equation}
 and $L_{\CMcal{I}}$ is some function on $U_{\CMcal{I}} = \{x \in M : \CMcal{I} \neq 0\}$ at $x$. In particular, if $L_{\CMcal{R}}$ is constant then $\mathcal{M}$ is called a pseudo-symmetric manifold of constant type \cite{NHMS}.

\begin{thm}\label{7}
If a non-flat (2n + 1)-dimensional $N(k)$-paracontact metric manifold $\mathcal{M}$ is a proper pseudo-symmetric manifold then the manifold is a pseudo-symmetric manifold of constant type.
\end{thm}
\begin{proof}
If $\mathcal{M}$ is a Desczc type pseudo-symmetric then from (\ref{2.6}) and (\ref{I2}), one can easily obtain that $\mathcal{R}(\xi, X)\cdot \mathcal{R} =k\{(\xi \wedge X)\cdot \mathcal{R}\}$, which specifies that the pseudo-symmetry function $L_{\mathcal{R}} =k$ (a constant). Hence, the manifold $\mathcal{M}$ is a pseudo-symmetric manifold of constant type. This completes the proof.
\end{proof}

\section{\bf Ricci generalized pseudo-symmetric $N(k)$-paracontact metric manifolds \label{sec5}}
If Ricci curvature tensor $\mathcal{S}$ holds
\begin{equation}
\mathcal{R}(X, Y)\cdot \mathcal{R}=L[(X\wedge_{\mathcal{S}} Y)\cdot \mathcal{R}], \label{II1}
\end{equation}
at every point of the Riemannian manifold then the manifold $M$ is called Ricci
generalized pseudo-symmetric manifold. Where $(X\wedge_{\mathcal{S}} Y)Z$ is given by
\begin{equation}
\label{II2}   (X\wedge_{\mathcal{S}} Y)Z= \mathcal{S}(Y, Z)X - \mathcal{S}(X, Z)Y,
\end{equation}
 and $L$ is some function.
\begin{thm}\label{7}
 A non-flat (2n + 1)-dimensional $N(k)$-paracontact metric manifold $\mathcal{M}$ holds $\mathcal{R}(\xi, X)\cdot R = L\{(\xi\wedge_{\mathcal{S}} X)\cdot R\}$ then $\mathcal{M}$ is either semi-symmetric or $k=0$ or Einstein manifold.
\end{thm}

\begin{proof}
Assume that $\mathcal{M}$ satisfies $g((\mathcal{R}(\xi, X)\cdot \mathcal{R}(Y, Z)W), \xi) = L(g(((\xi\wedge_{\mathcal{S}} X)\cdot R)(Y, Z)W), \xi)$. Then we have
\begin{eqnarray}
\nonumber g(\mathcal{R}(\xi, X)\mathcal{R}(Y,Z)W, \xi) - g(\mathcal{R}(\mathcal{R}(\xi, X)Y, Z)W, \xi) - g(\mathcal{R}(Y, \mathcal{R}(\xi, X)Z)W, \xi)\\
\nonumber - g(\mathcal{R}(Y, Z)\mathcal{R}(\xi, X)W, \xi)= g((\xi\wedge_{\mathcal{S}} X)\mathcal{R}(Y, Z)W, \xi) - g(\mathcal{R}((\xi\wedge_{\mathcal{S}}X)Y, Z)W, \xi) \\
 - g(\mathcal{R}(Y, (\xi\wedge_{\mathcal{S}}X)Z)W, \xi) - g(\mathcal{R}(Y, Z)(\xi\wedge_{\mathcal{S}}X)W, \xi). \label{8.1}
\end{eqnarray}

Using (\ref{2.6}), (\ref{2.10}) and (\ref{II2}), we get
\begin{equation}
Lk\{2nkg(X, Z)-S(X, Z)\}=0. \label{8.2}
\end{equation}
Hence the proof

\end{proof}

\section{\bf Generalized Ricci recurrent $N(k)$-paracontact metric manifolds \label{sec6}}
\begin{defi}
A $N(k)$-paracontact metric manifold $\mathcal{M}$ is said to be generalized Ricci recurrent if its non vanishing Ricci tensor $\mathcal{S}$ satisfies the condition
\begin{equation}
(\nabla_{X}\mathcal{S})(Y, Z) = \mathcal{A}(X)\mathcal{S}(Y, Z) + \mathcal{B}(X) g(Y, Z), \label{5.1}
\end{equation}
where $\mathcal{A}$ and $\mathcal{B}$ are two non-zero 1-forms such that $\mathcal{A}(X)=g(X, \zeta_{1})$ and $\mathcal{B}(X)=g(X, \zeta_{2})$, $\zeta_{1}$ and $\zeta_{2}$ being the associated vector fields of the 1-forms.
\end{defi}

From the preliminaries of $N(k)$-paracontact metric manifold, one can easily get
\begin{equation}
(\nabla_{X}\mathcal{S})(Y, \xi) = 2nk\{g(X, \phi Y) - g(hX, \phi Y) + \mathcal{S}(Y, \phi X) - \mathcal{S}(Y, \phi hX). \label{5.1a}
\end{equation}
Taking $Z=\xi$ in (\ref{5.1}) and using (\ref{5.1a}), we have
\begin{equation}
2nk\{g(X, \phi Y) - g(hX, \phi Y) + \mathcal{S}(Y, \phi X) - \mathcal{S}(Y, \phi hX) = \{2nk \mathcal{A}(X) + \mathcal{B}(X)\} \eta(Y). \label{5.2}
\end{equation}

On substituting $Y$ by $\xi$ in the above equation, one can obtain
\begin{equation}
 2nk \mathcal{A}(X) + \mathcal{B}(X)=0. \label{5.3}
\end{equation}

Thus we have:
\begin{thm}\label{T5}
 In a generalized Ricci recurrent $N(k)$-paracontact metric manifold $\mathcal{M}$ the associated 1-forms are linearly dependent and the vector fields of the associated 1-forms are of opposite direction.
\end{thm}

\section{\bf Examples \label{sec7}}
In this section we show that the existence of generalized Ricci recurrent 3-dimensional $N(k)$-paracontact metric manifold, which verifies the result of section \ref{sec6}. We consider a 3-dimensional manifold $\mathcal{M} = \{(x, y, z)\in R^3, (x, y, z)\neq 0\}$, where $(x, y, z)$ are the standard coordinate in $\mathcal{R}^3$. Let $e_{1}$, $e_{2}$, $e_{3}$ be three linearly independent vector fields in $\mathcal{M}$ which satisfies
\begin{equation}
\nonumber e_{1}= \frac{\partial}{\partial{x}}+z\frac{\partial}{\partial{y}}-2y\frac{\partial}{\partial{z}}, \,\,\,\,\,\,\,\,\,\,\, e_{2}= \frac{\partial}{\partial{y}}, \,\,\,\,\,\,\,\,\,\, e_{3}= \frac{\partial}{\partial{z}}.
\end{equation}

We define the pseudo-Riemannian metric $g$ as follows $g(e_{1}, e_{2}) = g(e_{3}, e_{3}) = 1$ and $g(e_{i}, e_{j}) = 0,$ otherwise.
We obtain
\begin{equation}
\nonumber [e_{1}, e_{2}]=2e_{3}, \,\,\,\,\,\, [e_{1}, e_{3}]=-e_{2} , \,\,\,\,\, [e_{2}, e_{3}]=0.
\end{equation}

We consider $\eta= 2ydx + dz$ and satisfying $\eta(e_{1}) = 0 = \eta(e_{2}), \eta(e_{3}) = 1.$ Let $\phi$ be the $(1, 1)$-tensor field defined by $\phi (e_{1})= e_{1}, \,\,\,\,\,\, \phi (e_{2})= -e_{2}, \,\,\,\,\, \phi (e_{3})= 0$. Then we
have $d\eta(e_{1}, e_{2}) = g(e_{1}, \phi e_{2}), d\eta(e_{1}, e_{3}) = g(e_{1}, \phi e_{3})$ and $d\eta(e_{2}, e_{3}) = g(e_{2}, \phi e_{3})$

Thus for $e_{3} = \xi$, the structure $(\phi, \xi, \eta, g)$ is a paracontact metric structure on
$\mathcal{M}$ with
\[ he_{i}=
\begin{cases}
      e_{2} & for \,\,\,\,\, i=1 \\
      0 & for \,\,\,\,\,\,\, i=2, 3\\
   \end{cases}
\]

Using Koszul's formula, we can easily calculate
\begin{equation}
\nonumber \left(
  \begin{array}{ccc}
   \nabla_{E_{1}}E_{1} &\nabla_{E_{1}}E_{2} &\nabla_{E_{1}}E_{3} \\
   \nabla_{E_{2}}E_{1} &\nabla_{E_{2}}E_{2} &\nabla_{E_{2}}E_{3}  \\
   \nabla_{E_{3}}E_{1} &\nabla_{E_{3}}E_{2} &\nabla_{E_{3}}E_{3} \\
  \end{array}
\right) =\left(
           \begin{array}{ccc}
             e_{3} & e_{3} & -e_{1}-e_{2} \\
             -e_{3} & 0 & e_{2}  \\
              -e_{1}& e_{2} & 0 \\
           \end{array}
         \right).
\end{equation}

So, above relations tells us that the manifold satisfies the equation (\ref{2.3}) for any vector field $X$ in $\chi(\mathcal{M})$ and $\xi=e_{3}$. Hence the manifold is a paracontact metric manifold.\\
 Using the above relations it can be verified that\\
 \begin{tabular}{ccc}
 $\mathcal{R}(e_{1}, e_{2})e_{3} = 0$, & \,\,\,\,\,\,\,\,\,\,$\mathcal{R}(e_{2}, e_{3})e_{3} =-e_{2}$, & \,\,\,\,\,\,\,\,\,\,\,\,$\mathcal{R}(e_{1}, e_{3})e_{3} = 2e_{2}-e_{1}$,\\
 \,\,\,\,\,\,\,\,\,\,$\mathcal{R}(e_{1}, e_{2})e_{2} = -3e_{2}$, & \,\,\,\,$\mathcal{R}(e_{2}, e_{3})e_{2} = 0$, & $\mathcal{R}(e_{1}, e_{3})e_{2} =e_{3}$,\\
 \,\,\,\,$\mathcal{R}(e_{1}, e_{2})e_{1} = 3e_{1}$, & \,\,\,\,\,\,\,\,$\mathcal{R}(e_{2}, e_{3})e_{1} =e_{3}$, & \,\,\,\,\,\,\,\,$\mathcal{R}(e_{1}, e_{3})e_{1} =-2e_{3}$. \\
 \end{tabular}

In view of the expressions of the curvature tensors we conclude that the manifold is a $N(k)$-paracontact metric manifold with $k=-1$.

Using this, we find the values of the Ricci tensor as follows
\begin{equation}
\nonumber \mathcal{S}(e_{1}, e_{1}) = -1, \,\,\,\,\,\,\,\,\,\, \mathcal{S}(e_{2}, e_{2}) =-3  , \,\,\,\,\,\,\,\,\,\, \mathcal{S}(e_{3}, e_{3}) = 2.
\end{equation}

Since $\{e_{1}, e_{2}, e_{3}\}$ forms a basis of $\mathcal{M}^3$, any vector fields $X, Y \in \chi(\mathcal{M})$ can be
written as $X = a_{1}e_{1} + b_{1}e_{2} + c_{1}e_{3}$ and $Y = a_{2}e_{1} + b_{2}e_{2} + c_{2}e_{3}$, where $a_{i}, b_{i}, c_{i} \in \mathcal{R}^{+}$ (the set of all positive real numbers), $i = 1, 2$. This implies that
\begin{equation}
\nonumber \mathcal{S}(X, Y) = -a_{1}b_{1}-3a_{2}b_{2}+2a_{3}b_{3}\,\,\,\,\,\,\,\,\,\,\,\, and\,\,\,\,\,\,\,\,\,\,\, g(X, Y)=a_{1}b_{2} + a_{3}b_{3}.
\end{equation}

By virtue of above, we have the following:
\begin{eqnarray}
\nonumber & & (\nabla_{e_{1}}\mathcal{S})(X, Y)=-\{3(a_{3}b_{1} + a_{1}b_{3}) + 5(a_{3}b_{2}+a_{2}b_{3})\},\\
\nonumber & & (\nabla_{e_{2}}\mathcal{S})(X, Y)=2(a_{1}b_{3} + a_{3}b_{1}) + 3(a_{3}b_{2}+a_{2}b_{3}),\\
\nonumber & & (\nabla_{e_{3}}\mathcal{S})(X, Y)=-(2a_{1}b_{1}-6a_{2}b_{2}).
\end{eqnarray}

This means that manifold under the consideration is not Ricci symmetric. Let us now consider the 1-forms
\begin{eqnarray}
\nonumber & & \mathcal{A}(e_{1})=\frac{\{3(a_{3}b_{1} + a_{1}b_{3}) + 5(a_{3}b_{2}+a_{2}b_{3})\}}{a_{1}b_{1}+3a_{2}b_{2}+2a_{1}b_{2}},\,\,\,\,\,\,\,\mathcal{B}(e_{1})=-2\frac{\{3(a_{3}b_{1} + a_{1}b_{3}) + 5(a_{3}b_{2}+a_{2}b_{3})\}}{a_{1}b_{1}+3a_{2}b_{2}+2a_{1}b_{2}},\\
\nonumber & & \mathcal{A}(e_{2})= \frac{\{2(a_{1}b_{3} + a_{3}b_{1}) + 3(a_{3}b_{2}+a_{2}b_{3})\}}{a_{1}b_{1}+3a_{2}b_{2}+2a_{1}b_{2}},\,\,\,\,\,\,\, \mathcal{B}(e_{2})= -2\frac{\{2(a_{1}b_{3} + a_{3}b_{1}) + 3(a_{3}b_{2}+a_{2}b_{3})\}}{a_{1}b_{1}+3a_{2}b_{2}+2a_{1}b_{2}}, \\
\nonumber & & \mathcal{A}(e_{3})= \frac{\{(2a_{1}b_{1}-6a_{2}b_{2}))\}}{a_{1}b_{1}+3a_{2}b_{2}+2a_{1}b_{2}},\,\,\,\,\,\,\,\,\,\,\,\,\,\,\,\,\,\,\,\,\,\,\,\,\,\,\,\,\,\,\,\,\,\,\,\,\,\,\,\, \mathcal{B}(e_{3})= -2\frac{\{(2a_{1}b_{1}-6a_{2}b_{2}))\}}{a_{1}b_{1}+3a_{2}b_{2}+2a_{1}b_{2}},
\end{eqnarray}
at any point $X\in M$. From (\ref{5.1}) we have
\begin{equation}
(\nabla_{e_{i}}\mathcal{S})(Y, Z) = \mathcal{A}(e_{i})\mathcal{S}(Y, Z) + 3\mathcal{B}(e_{i}) g(Y, Z), \,\,\,\,\,\,\,\,\,\,\,\, i=1, 2, 3.\label{6.1}
\end{equation}

It can be easily shown that the manifold with the above 1-forms satisfies the relation (\ref{6.1}). Hence the manifold under consideration is a generalized Ricci recurrent $N(k)$-paracontact metric manifold. Also with the help of these 1-forms we can easily verify the theorem \ref{T5} for three dimensional case.

\end{document}